\documentclass{elsarticle}
\usepackage{amsthm}
\usepackage{amssymb}
\usepackage{amsmath}
\usepackage{enumerate}

\newtheorem{thm}{Theorem}
\newtheorem{lem}[thm]{Lemma}
\newtheorem{cor}[thm]{Corollary}
\newtheorem{defi}[thm]{Definition}
\newtheorem{quest}[thm]{Question}
\newtheorem{prop}[thm]{Proposition}

\begin{document}

\title{Which traces are spectral?}

\author{F. Sukochev}
\ead{f.sukochev@unsw.edu.au}
\address{School of Mathematics and Statistics, University of New South Wales, Sydney, 2052, Australia.}
\author{D. Zanin}
\ead{d.zanin@unsw.edu.au}
\address{School of Mathematics and Statistics, University of New South Wales, Sydney, 2052, Australia.}

\begin{abstract} Among ideals of compact operators on a Hilbert space we identify a subclass of those closed with respect to the logarithmic submajorization.  Within this subclass, we answer the questions asked by Pietsch \cite{Pietsch_nachrichten} and by Dykema, Figiel, Weiss and Wodzicki \cite{DFWW}. In the first case, we show that Lidskii-type formulae hold for every trace on  such ideal. In the second case, we provide the description of the commutator subspace associated with a given ideal. Finally, we prove that a positive trace on an arbitrary ideal is spectral if and only if it is monotone with respect to the logarithmic submajorization.
\end{abstract}


\date{}

\bibliographystyle{plain}

\maketitle

\section{Introduction}

Let $H$ be a separable Hilbert space and let $\mathcal{L}(H)$ be the algebra of all bounded operators on $H.$ The set $\mathcal{L}_1$ of all trace class operators is an ideal in $\mathcal{L}(H).$ It carries a special functional --- the classical trace ${\rm Tr}.$ There is also the description of ${\rm Tr}$ as the sum of eigenvalues.
\begin{equation}\label{original lidskii}
{\rm Tr}(T)=\sum_{n=0}^{\infty}\lambda(n,T),\quad T\in\mathcal{L}_1.
\end{equation}
Here, $\lambda(T)=\{\lambda(n,T)\}_{n\geq0}$ is the sequence of eigenvalues\footnote{Non-zero eigenvalues are repeated according to their algebraic multiplicity and arranged so that $\{|\lambda(n,T)|\}_{n\geq0}$ is a decreasing sequence. If there are only finitely many (or none) non-zero eigenvalues, then all the other components of $\lambda(T)$ are zeros.} of a compact operator $T.$ This result was shown by von Neumann in \cite{vNeumann1932} for self-adjoint operators and then by Lidskii in \cite{Lidskii1959} in general case. Formula \eqref{original lidskii} is now known as Lidskii formula (see e.g. \cite{Simon}).

Fix an orthonormal basis in the Hilbert space $H.$ The subalgebra of $\mathcal{L}(H)$ consisting of all diagonal operators with respect to this basis is naturally isomorphic to the algebra $l_{\infty}$ of all bounded complex sequences. Further, we always identify the algebra $l_{\infty}$ with this diagonal subalgebra. Thus, the notations $x\in\mathcal{L}(H)$ (or $x\in\mathcal{I}$ for some ideal $\mathcal{I}$ in $\mathcal{L}(H)$) make perfect sense for an element $x\in l_{\infty}.$

Identifying the sequence $\lambda(T)$ with an element of $\mathcal{L}(H),$ we can write Lidskii formula as ${\rm Tr}(T)={\rm Tr}(\lambda(T))$ for all $T\in\mathcal{L}_1.$ A natural question concerning the extension of this formula to other ideals and traces on these ideals has been treated in a number of publications (see e.g. \cite{AS,BF,DK,Fackcomm2004,KWsurvey,KLPS,LSZ,Pietsch_nachrichten,SSZ}). In what follows, $\mathcal{I}$ is an ideal in $\mathcal{L}(H)$ and $\varphi$ is a trace on $\mathcal{I},$ i.e. a linear functional $\varphi:\mathcal{I}\to\mathbb{C}$ satisfying the condition
$$\varphi(AB)=\varphi(BA),\quad A\in\mathcal{I},B\in\mathcal{L}(H).$$

The following problem was stated by Pietsch (see p.9 in \cite{Pietsch_nachrichten}).

\begin{quest}\label{spectral question} For which traces $\varphi$ on an ideal $\mathcal{I}$ do we have
\begin{equation}\label{general lidskii}
\varphi(T)=\varphi(\lambda(T)),\quad T\in\mathcal{I}?
\end{equation}
\end{quest}

A given trace $\varphi$ on the ideal $\mathcal{I}$ satisfying \eqref{general lidskii} is called spectral.

Study of traces in general and Question \ref{spectral question} in particular are closely related to the description of the commutator subspace of an ideal $\mathcal{I}$ in $\mathcal{L}(H).$ The latter subspace (denoted by ${\rm Com}(\mathcal{I})$) is a linear span of the elements $AB-BA,$ $A\in\mathcal{I},$ $B\in\mathcal{L}(H).$ The following question was asked in \cite{DFWW} (see also \cite{DFWW2}).

\begin{quest}\label{commutator question} Does the commutator subspace admit a description in spectral terms?
\end{quest}

Note that, for an operator $T\in\mathcal{I},$ we have $T\in{\rm Com}(\mathcal{I})$ if and only if all traces on $\mathcal{I}$ vanish on $T.$ Thus, if Question \ref{spectral question} is answered in positive (in a sense that all traces on $\mathcal{I}$ are spectral) then, for $T\in\mathcal{I},$ we have $T\in{\rm Com}(\mathcal{I})$ if and only if $\lambda(T)\in{\rm Com}(\mathcal{I}).$ Hence, a positive answer to Question \ref{spectral question} implies a positive answer to Question \ref{commutator question} and vice versa.

For normal operators, Question \ref{commutator question} was answered in the affirmative in \cite{Kalton1998} (see Theorem 3.1 there) and in \cite{DFWW} (see Theorem 5.6 there) for arbitrary ideals.

\begin{thm}\label{normal commutator} A normal operator $N\in\mathcal{I}$ belongs to ${\rm Com}(\mathcal{I})$ if and only if $C\lambda(N)\in\mathcal{I}.$
\end{thm}

Here, $C:l_{\infty}\to l_{\infty}$ is Cesaro operator defined by
$$Cx=(x(0),\frac{x(0)+x(1)}{2},\frac{x(0)+x(1)+x(2)}{3},\cdots),\quad x=(x(0),x(1),x(2),\cdots)\in l_{\infty}.$$

However, if the operator $T\in\mathcal{I}$ is not normal, the situation becomes substantially more complicated. The following partial answer to Question \ref{commutator question} was obtained by Kalton (see Theorem 3.3. in \cite{Kalton1998}).

\begin{thm}\label{kalton result} Let an ideal $\mathcal{I}$ be geometrically stable. An operator $T\in\mathcal{I}$ belongs to ${\rm Com}(\mathcal{I})$ if and only if $C\lambda(T)\in\mathcal{I}.$
\end{thm}

Recall that an ideal $\mathcal{I}$ is called geometrically stable \cite{Kalton1998} if we have
$$\{(\prod_{m=0}^n\mu(k,T))^{1/(n+1))}\}_{n\geq0}\in\mathcal{I},\quad T\in\mathcal{I}.$$
Here, $\mu(T)=\{\mu(n,T)\}_{n\geq0}$ is the sequence of singular values of a compact operator $T,$ that is, the sequence $\lambda(|T|).$

Motivated by Theorem \ref{kalton result}, the authors of \cite{DFWW2} asked whether the same assertion holds in an arbitrary ideal (see Problem 5.1 in \cite{DFWW2}). In order to answer this question and to extend Kalton's result, we need a concept of logarithmic submajorization\footnote{Although the logarithmic submajorization technique was used already by Weyl in \cite{Weyl_logmaj}, it took a while until Ando and Hiai formally defined this order in \cite{andohiai1994} and applied it to the operator inequalities in \cite{andohiai2011}.} and the class of ideals closed with respect to the latter.

\begin{defi} If $A,B\in\mathcal{L}(H),$ then the operator $B$ is logarithmically submajorized by the operator $A$ (written $B\prec\prec_{\log} A$) if
$$\prod_{k=0}^n\mu(k,B)\leq\prod_{k=0}^n\mu(k,A),\quad n\geq0.$$
\end{defi}

\begin{defi} An ideal $\mathcal{I}$ is said to be closed with respect to the logarithmic submajorization if $B\prec\prec_{\log} A\in\mathcal{I}$ implies that $B\in\mathcal{I}.$
\end{defi}

Every geometrically stable ideal is closed with respect to the logarithmic submajorization (see Lemma \ref{gs implies cl} below). The converse assertion is shown to be true for countably generated ideals by Dykema and Kalton (see Proposition 1.1 and Theorem 1.3  in \cite{DykemaKalton1998}). However, this is not the case for general ideals. Indeed, we show (see Theorem \ref{dk example} below) that the class of ideals closed with respect to the logarithmic submajorization is strictly wider than that of geometrically stable ideals.

Our first main result extends Theorem \ref{kalton result} to a wider class of ideals and answers Question \ref{commutator question} (or Problem 5.1 in \cite{DFWW2}) in the setting of ideals closed with respect to the logarithmic submajorization.

\begin{thm}\label{main commutator theorem} Let an ideal $\mathcal{I}$ be closed with respect to the logarithmic submajorization and let $T\in\mathcal{I}.$ We have $T\in{\rm Com}(\mathcal{I})$ if and only if $C\lambda(T)\in\mathcal{I}.$
\end{thm}

Theorem \ref{main commutator theorem} implies the following Lidskii-type result on traces (in particular, the first part of Theorem \ref{extended kalton} provides an alternative\footnote{For discussions of various approaches to proof of \eqref{original lidskii}, we refer the reader to Chapter 3 in \cite{Simon} } proof of \eqref{original lidskii}). Furthermore, the second part of Theorem \ref{extended kalton} shows that the class of ideals closed with respect to the logarithmic submajorization is optimal for Lidskii formulae \eqref{general lidskii}.

\begin{thm}\label{extended kalton} Let $\mathcal{I}$ be an ideal in $\mathcal{L}(H)$ and let $\varphi$ be a trace on $\mathcal{I}.$
\begin{enumerate}[(a)]
\item\label{kalton first} If $\mathcal{I}$ is closed with respect to the logarithmic submajorization, then $\varphi(T)=\varphi(\lambda(T))$ for all $T\in\mathcal{I}.$
\item\label{kalton second} If $\mathcal{I}$ is not closed with respect to the logarithmic submajorization, then there exists an operator $T\in\mathcal{I}$ such that $\lambda(T)\notin\mathcal{I}.$ In particular, the equality \eqref{general lidskii} makes no sense in this case.
\end{enumerate}
\end{thm}

In applications, especially in noncommutative geometry, one is mostly interested in positive traces (see e.g. \cite{BF,Connes,Dixmier,Fackcomm2004,LSZ,Takesaki}). We refer the reader to the papers \cite{AS,SSZ} and to the book \cite{LSZ} for the treatment of Lidskii formula for positive traces. For such traces, we are able to present a complete solution of Question \ref{spectral question}: a positive trace on an ideal $\mathcal{I}$ is spectral if and only if it is monotone with respect to the logarithmic submajorization.

\begin{defi} Let $\mathcal{I}$ be an ideal in $\mathcal{L}(H)$ and let $\varphi$ be a positive trace on $\mathcal{I}.$ The trace $\varphi$ is said to be monotone with respect to the logarithmic submajorization if $\varphi(B)\leq\varphi(A)$ for all $0\leq A,B\in\mathcal{I}$ with $B\prec\prec_{\log} A.$
\end{defi}

Later we show (see Lemma \ref{geom envelope lemma}) that for every ideal $\mathcal{I},$ there exists the least ideal\footnote{LE stands for the \lq\lq logarithmic envelope\rq\rq.} $LE(\mathcal{I})$ which contains $\mathcal{I}$ and which is closed with respect to the logarithmic submajorization.

Our second main result can be read as follows.

\begin{thm}\label{main theorem} Let $\mathcal{I}$ be an ideal in $\mathcal{L}(H)$ and let $\varphi$ be a positive trace on $\mathcal{I}.$
\begin{enumerate}[(a)]
\item\label{main first} If $\varphi(T)=\varphi(\lambda(T))$ for all $T\in\mathcal{I}$ with $\lambda(T)\in\mathcal{I},$ then the trace $\varphi$ is monotone with respect to the logarithmic submajorization.
\item\label{main second} If $\varphi$ is monotone with respect to the logarithmic submajorization, then $\varphi$ extends to a positive trace on $LE(\mathcal{I}).$ In particular, $\varphi(T)=\varphi(\lambda(T))$ for all $T\in\mathcal{I}.$
\end{enumerate}
\end{thm}

Theorem \ref{main theorem} leads to the following surprising corollary. While a positive trace may not necessarily be monotone with respect to the Hardy-Littlewood submajorization (see \cite{KScanada,LSZ,SZcrelle}), it is always monotone with respect to the logarithmic submajorization (on the ideals closed with respect to the latter).

\begin{cor}\label{cl cor} If an ideal $\mathcal{I}$ is closed with respect to the logarithmic submajorization (e.g. in quasi-Banach ones), then every positive trace on $\mathcal{I}$ is monotone with respect to the logarithmic submajorization.
\end{cor}

We should note, however, that the converse assertion to Corollary \ref{cl cor} is false. In Theorem \ref{dk example} we present a principal ideal which is not closed with respect to the logarithmic submajorization and such that every positive trace on it is monotone with respect to the logarithmic submajorization (and, therefore, spectral by Theorem \ref{main theorem}).

\section{Preliminaries}

\subsection{Eigenvalues and singular values}

For all compact operators $A,B\in\mathcal{L}(H),$ we have (see e.g. Corollary 2.3.16 in \cite{LSZ})
\begin{equation}\label{mu sum}
\mu(A+B)\leq\sigma_2(\mu(A)+\mu(B)).
\end{equation}
Here, the dilation operator $\sigma_n:l_{\infty}\to l_{\infty}$ is defined as follows.
$$\sigma_nx=(\underbrace{x(0),\cdots,x(0)}_{\mbox{$n$ times}},\underbrace{x(1),\cdots,x(1)}_{\mbox{$n$ times}},\cdots),\quad x=(x(0),x(1),\cdots)\in l_{\infty}.$$

The following inequality relating eigenvalues with singular values is due to Weyl (see Lemma II.3.3. in \cite{GohKr1} or Weyl's original paper \cite{Weyl_logmaj}).
\begin{thm}\label{weyl theorem} For every compact operator $T,$ we have
\begin{equation}\label{lambda logmaj mu}
\prod_{k=0}^n|\lambda(k,T)|\leq\prod_{k=0}^n\mu(k,T),\quad n\geq0.
\end{equation}
Equivalently, $\lambda(T)\prec\prec_{\log}\mu(T).$
\end{thm}
The following assertion due to Dykema and Kalton \cite{DykemaKalton1998} may be viewed as a converse to Theorem \ref{weyl theorem}.
\begin{lem}\label{dk prop11} If $x=\mu(x)\in c_0$ and $|y|=\mu(y)\in c_0$ are such that $y\prec\prec_{\log}x,$ then there exists a compact operator $T\in\mathcal{L}(H)$ such that $\lambda(T)=y$ and $\mu(T)\leq x.$
\end{lem}

A class of all positive traces on an ideal $\mathcal{I}$ admits an equivalent description in terms of singular value sequences\footnote{A variant of Lemma \ref{trace monotone} for symmetrically normed ideals can be found in Lemma 2.7.4 of \cite{LSZ}}.
\begin{lem}\label{trace monotone} Let $\mathcal{I}$ be an ideal in $\mathcal{L}(H)$ and let $\varphi:\mathcal{I}\to\mathbb{C}$ be a linear mapping. The functional $\varphi$ is a positive trace if and only if $\varphi(B)\leq\varphi(A)$ for all $0\leq A,B\in\mathcal{I}$ such that $\mu(B)\leq\mu(A).$
\end{lem}
\begin{proof} Let $\varphi$ be a positive trace. First, we prove that $\varphi(A)=\varphi(\mu(A))$ for every positive $A\in\mathcal{I}.$ Indeed, there exists an isometry $U\in\mathcal{L}(H)$ such that $A=U\mu(A)U^*$ and $U^*U=1.$ Therefore,
$$A-\mu(A)=U\mu(A)\cdot U^*-U^*\cdot U\mu(A)\in{\rm Com}(\mathcal{I}).$$
In particular, we have $\varphi(A)=\varphi(\mu(A)).$

If now $A,B\in\mathcal{I}$ are positive operators such that $\mu(B)\leq\mu(A),$ then
$$\varphi(B)=\varphi(\mu(B))\leq\varphi(\mu(A))=\varphi(A).$$
This proves necessity.

We now prove sufficiency. Let $0\leq A\in\mathcal{I}$ and let $U\in\mathcal{L}(H)$ be a unitary operator. We have $\mu(U^{-1}AU)=\mu(A),$ which yields (see the assumption) $\varphi(U^{-1}AU)=\varphi(A).$ By linearity, the latter equality holds for every $A\in\mathcal{I}$ (i.e., without assumption of positivity). Substituting $UA$ instead of $A,$ we obtain $\varphi(AU)=\varphi(UA).$ Since every operator $B\in\mathcal{L}(H)$ is a linear combination of $4$ unitaries, it follows that $\varphi(AB)=\varphi(BA)$ for every $A\in\mathcal{I}$ and every $B\in\mathcal{L}(H).$ Thus, $\varphi$ is a trace. The assertion that $\varphi\geq0$ is immediate.
\end{proof}

The following lemma (used only in Section \ref{Examples} below) demonstrates that every positive trace on the \lq\lq non-summable\rq\rq~ideals is singular.

\begin{lem}\label{singular} Let $\mathcal{I}\not\subset\mathcal{L}_1$ and let $\varphi$ be a positive trace on $\mathcal{I}.$
\begin{enumerate}[(a)]
\item\label{singa} The trace $\varphi$ is singular, that is $\varphi$ vanishes on the finite rank operators.
\item\label{singb} If $A,B\in\mathcal{I}$ are such that $\mu(k,B)=o(\mu(k,A))$ as $k\to\infty,$ then $\varphi(B)=0.$
\end{enumerate}
\end{lem}
\begin{proof} Select a positive operator $A\in\mathcal{I}$ such that $A\notin\mathcal{L}_1.$ Let $P_k,$ $k\geq0,$ be rank one projections such that
$$A=\sum_{k=0}^{\infty}\mu(k,A)P_k.$$
Here, the convergence is taken in the strong operator topology. For every $n\geq0,$ it follows from Lemma \ref{trace monotone} that
$$\varphi(A)\geq\varphi(\sum_{k=0}^n\mu(k,A)P_k)=\sum_{k=0}^n\mu(k,A)\varphi(P_k)=\sum_{k=0}^n\mu(k,A)\varphi(P_0).$$
If $\varphi(P_0)\neq0,$ then, the sum at the right hand side tends to $\infty$ as $n\to\infty.$ This contradicts to the assumption $\varphi(A)<\infty.$

Hence, $\varphi$ vanishes on every finite rank projection. Since every finite rank operator is a linear combination of finite rank projections, the assertion  \eqref{singa} follows.

In order to prove \eqref{singb}, we may assume without loss of generality that $B\geq0.$ Fix $\varepsilon>0$ and select $N$ such that, for every $n\geq N$ we have $\mu(n,B)\leq\varepsilon\mu(n,A).$ In particular, $\mu(B)\leq\varepsilon\mu(A)+\mu(B)\chi_{[0,N)}.$ Using Lemma \ref{trace monotone} and part \eqref{singa}, we infer that $\varphi(B)\leq\varepsilon\varphi(A).$ Since $\varepsilon$ is arbitrarily small, the assertion \eqref{singb} follows.
\end{proof}

Since Hilbert spaces $H^{\oplus n}$ and $H$ are isometrically isomorphic, it is often convenient to identify them so that the operator $A_1\oplus\cdots\oplus A_n,$ $A_k\in\mathcal{L}(H),$ $1\leq k\leq n,$ also belongs to $\mathcal{L}(H).$ We note the equalities
$$\lambda(A^{\oplus n})=\sigma_n\lambda(A),\quad \mu(A^{\oplus n})=\sigma_n\mu(A),\quad n\geq1,$$
which are frequently used in the text. For every trace on $\mathcal{I}$ and for every $x\in l_{\infty}$ such that $x\in\mathcal{I}$ we have $\varphi(\sigma_nx)=\varphi(x^{\oplus n})=n\varphi(x),$ $n\geq1.$

\subsection{Ringrose theorem}

An operator $N\in\mathcal{L}(H)$ is said to be normal if $NN^*=N^*N.$ A compact operator $Q\in\mathcal{L}(H)$ is said to be quasi-nilpotent if $\lambda(Q)=0.$ The following result belongs to Ringrose (see Chapter 4 in \cite{Ringrose1971}).

\begin{thm}\label{ringrose theorem} For every compact operator $T\in\mathcal{L}(H),$ there exists a compact normal operator $N\in\mathcal{L}(H)$ and compact quasi-nilpotent operator $Q\in\mathcal{L}(H)$ such that $T=N+Q$ and $\lambda(T)=\lambda(N).$
\end{thm}

\subsection{Hardy-Littlewood submajorization} The following submajorization was introduced by Hardy and Litttlewood. We refer the reader to the book \cite{LSZ} for details.

\begin{defi}\label{majorization def} Let $0\leq A,B\in\mathcal{L}(H).$ We say that the operator $B$ is submajorized by the operator $A$ (written $B\prec\prec A$) if
$$\sum_{k=0}^{n-1}\mu(k,B)\leq\sum_{k=0}^{n-1}\mu(k,B),\quad n>0.$$
\end{defi}

One of the important features of Hardy-Littlewood submajorization is its nice behaviour with respect to the linear structure of $\mathcal{L}(H).$ For every compact positive $A,B\in\mathcal{L}(H),$ we have (see e.g. Theorems 3.3.3 and 3.3.4 in \cite{LSZ})
\begin{equation}\label{maj sum}
A+B\prec\prec \mu(A)+\mu(B)\prec\prec 2\sigma_{1/2}\mu(A+B).
\end{equation}
Here, $\sigma_{1/2}:l_{\infty}\to l_{\infty}$ is an operator defined as follows.
$$\sigma_{1/2}x=(\frac{x(0)+x(1)}{2},\frac{x(2)+x(3)}{2},\cdots),\quad x=(x(0),x(1),\cdots)\in l_{\infty}.$$

\subsection{Uniform submajorization}

The following definition, introduced originally in \cite{KS} (see also \cite{LSZ}), plays a major role in our treatment of traces.

\begin{defi}\label{uniform majorization def} Let $0\leq A,B\in\mathcal{L}(H).$ We say that that the operator $B$ is uniformly submajorized by the operator $A$ (written $B\lhd A$) if there exists $\lambda\in\mathbb{N}$ such that
$$\sum_{k=\lambda m}^{n-1}\mu(k,B)\leq\sum_{k=m}^{n-1}\mu(k,B),\quad\lambda m<n.$$
\end{defi}

Uniform submajorization also behaves nicely with respect to the linear structure of $\mathcal{L}(H).$ One can strengthen the inequalities \eqref{maj sum} as follows. For every compact positive $A,B\in\mathcal{L}(H),$ we have (see e.g. Lemma 3.4.4 in \cite{LSZ})
\begin{equation}\label{uniform maj sum}
A+B\lhd \mu(A)+\mu(B)\lhd 2\sigma_{1/2}\mu(A+B).
\end{equation}

Uniform submajorization is a stronger condition than Hardy-Littlewood submajorization introduced in the preceding subsection. Theorem \ref{ks majorization theorem} below describes the convex hull of the set $\{0\leq z\in l_{\infty}:\ \mu(z)\leq\mu(x)\}$ in terms of uniform submajorization (see Theorem 3.4.2 in \cite{LSZ} or Theorem 5.4 in the original paper \cite{KS}).

\begin{thm}\label{ks majorization theorem} Let $0\leq x,y\in l_{\infty}.$
\begin{enumerate}[(a)]
\item If $y$ belongs to a convex hull of the set $\{0\leq z\in l_{\infty}:\ \mu(z)\leq\mu(x)\},$ then $y\lhd x.$
\item If $y\lhd x,$ then, for every $\varepsilon>0,$ the $(1-\varepsilon)y$ belongs to a convex hull of the set $\{0\leq z\in l_{\infty}:\ \mu(z)\leq\mu(x)\}.$
\end{enumerate}
\end{thm}

An importance of uniform submajorization for studying of traces may be seen from the following strengthening of Lemma \ref{trace monotone}.\footnote{A variant of Lemma \ref{trace monotone lhd} for symmetrically normed ideals was proved in Lemma 4.2.5 in \cite{LSZ}.}

\begin{lem}\label{trace monotone lhd} Let $\mathcal{I}$ be an ideal and let $\varphi:\mathcal{I}\to\mathbb{C}$ be a linear mapping. The functional $\varphi$ is a positive trace on $\mathcal{I}$ if and only if $\varphi(B)\leq\varphi(A)$ for all $0\leq A,B\in\mathcal{I}$ such that $B\lhd A.$
\end{lem}
\begin{proof} Let $\varphi$ be a positive trace on $\mathcal{I}$ and let $0\leq A,B\in\mathcal{I}$ be such that $B\lhd A.$ Fix $\varepsilon\in(0,1).$ By Theorem \ref{ks majorization theorem}, there exist $n\geq1,$ positive sequences $z_k\in c_0$ and positive constants $\lambda_k\in(0,1),$ $1\leq k\leq n$ such that $\mu(z_k)\leq\mu(A)$ for every $1\leq k\leq n,$ and such that
$$(1-\varepsilon)\mu(B)\leq\sum_{k=1}^n\lambda_kz_k,\quad \sum_{k=1}^n\lambda_k=1.$$
Observe that
$$(1-\varepsilon)\varphi(\mu(B))\leq\sum_{k=1}^n\lambda_k\varphi(z_k)$$
by the assumption and Lemma \ref{trace monotone}. Applying Lemma \ref{trace monotone} again, we infer that $\varphi(B)=\varphi(\mu(B))$ and $\varphi(z_k)\leq\varphi(A)$ for $1\leq k\leq n.$ Hence, $(1-\varepsilon)\varphi(B)\leq\varphi(A).$ Since $\varepsilon$ is arbitrarily small, it follows that $\varphi(B)\leq\varphi(A).$ This proves necessity.

In order to prove sufficiency, choose $0\leq A,B\in\mathcal{I}$ such that $\mu(B)\leq\mu(A).$ It follows from Definition \ref{uniform majorization def} that $B\lhd A$ and, therefore, by the assumption, $\varphi(B)\leq\varphi(A).$ An application of Lemma \ref{trace monotone} completes the proof of sufficiency.
\end{proof}

\section{Main technical inequalities}

The main result of this section is Theorem \ref{hardest estimate} below. It is the core component of our proof of Theorem \ref{extended kalton}.

Define a non-linear homogeneous operator $\mathbf{S}:l_{\infty}\to l_{\infty}$ by setting
\begin{equation}\label{s trans def}
(\mathbf{S}x)(k)=\mu(k,x)(1+\frac1{k+1}\log(\frac{\prod_{m=0}^k\mu(m,x)}{\mu(k,x)^{k+1}})),\quad n\geq0.
\end{equation}

The operator $\mathbf{S}$ is a technical device used in the next section to estimate the action of Cesaro operator $C$ on the eigenvalue sequences of real and imaginary parts of a quasi-nilpotent operator. In this section, we obtain main technical estimate for this operator.

\begin{lem}\label{s dec} For every $x\in l_{\infty},$ we have $\mathbf{S}x=\mu(\mathbf{S}x).$
\end{lem}
\begin{proof} Without loss of generality, we may assume that $x=\mu(x).$ We have to prove $(\mathbf{S}x)(k+1)\leq(\mathbf{S}x)(k),$ $k\geq0.$ First, note that, for every constant $C>0,$ the function
$$x\to x(1+\frac{k+1}{k+2}\log(\frac{C}{x})),\quad x\in[0,C]$$
is increasing. Fixing the values $x(0),\cdots,x(k),$ and setting $C=(\prod_{m=0}^kx(m))^{1/(k+1)},$ we infer that the function
$$x(k+1)\to (\mathbf{S}x)(k+1)=x(k+1)(1+\frac{k+1}{k+2}\log(\frac{C}{x(k+1)}))$$
is also increasing. Thus, for given $x(0),\cdots,x(k),$ the function $x(k+1)\to (\mathbf{S}x)(k+1)$ attains its maximal value when $x(k+1)$ takes its maximal value, which is $x(k).$ Therefore,
$$(\mathbf{S}x)(k+1)\leq x(k)(1+\frac{k+1}{k+2}\log(\frac{C}{x(k)}))=x(k)(1+\frac1{k+2}\log(\frac{\prod_{m=0}^kx(m)}{x(k)^{k+1}})).$$
It is immediate that the right hand side of the latter inequality does not exceed $(\mathbf{S}x)(k).$ This proves the assertion.
\end{proof}

\begin{lem}\label{ocenka s} If $x=\mu(x)\in l_{\infty},$ then
\begin{enumerate}[(a)]
\item\label{ocsa} For every $n\geq0$ and for every $k\geq n,$ we have
$$(\mathbf{S}x)(k)\leq x(n)(1+\frac1{k+1}\log(\frac{\prod_{m=0}^nx(m)}{x(n)^{n+1}})).$$
\item\label{ocsb} For every $n\geq0$ and for every $k\leq n,$ we have
$$(\mathbf{S}x)(k)\leq x(k)(1+\frac1{k+1}\log(\frac{\prod_{m=0}^nx(m)}{x(n)^{n+1}})).$$
\end{enumerate}
\end{lem}
\begin{proof} Despite the similarity of both estimates, they require essentially different proofs.
\begin{enumerate}[(a)]
\item Fixing the values $x(0),\cdots,x(k-1),$ and arguing as in the proof of Lemma \ref{s dec}, we infer that the function $x(k)\to (\mathbf{S}x)(k)$ is increasing. Hence,
$$(\mathbf{S}x)(k)\leq x(k-1)(1+\frac1{k+1}\log(\frac{\prod_{m=0}^{k-1}x(m)}{x(k-1)^k})).$$
Fixing the values $x(0),\cdots,x(k-2),$ and repeating the argument above, we infer that the function of the variable $x(k-1)\in[0,x(k-2)]$ standing at the right hand side of the preceding inequality is increasing. Hence,
$$(\mathbf{S}x)(k)\leq x(k-2)(1+\frac1{k+1}\log(\frac{\prod_{m=0}^{k-2}x(m)}{x(k-2)^{k-1}})).$$
Repeating the argument for $k-2,k-3,\cdots,n+1,$ we conclude the proof.
\item Since both sides of the inequality are homogeneous with respect to $x,$ we may assume without loss of generality that $x(n)=1.$ It follows that $x(k)\geq1,$ $1\leq k\leq n,$ and, therefore,
$$\frac{\prod_{m=0}^kx(m)}{x(k)^{k+1}}\leq\prod_{m=0}^kx(m)\leq\prod_{m=0}^nx(m)=\frac{\prod_{m=0}^nx(m)}{x(n)^{n+1}}.$$
Combining the preceding estimate with the Definition of $\mathbf{S}$ given in \eqref{s trans def} yields the assertion.
\end{enumerate}
\end{proof}

\begin{lem}\label{binomial} For every $u>0$ and for every positive integer $n,$ we have
$$\prod_{k=0}^{2n}(1+\frac{u}{k+1})\leq 2^{2n+u+2}.$$
\end{lem}
\begin{proof} Let $u\in[m-1,m]$ with $m\in\mathbb{N}.$ We have
$$\prod_{k=0}^{2n}(1+\frac{u}{k+1})\leq\prod_{k=0}^{2n}(1+\frac{m}{k+1})=\frac{\prod_{k=0}^{2n}(m+k+1)}{\prod_{k=0}^{2n}(k+1)}=$$
$$=\frac{(2n+m+1)!}{m!(2n+1)!}\leq 2^{2n+m+1}\leq 2^{2n+u+2}.$$
\end{proof}

\begin{thm}\label{hardest estimate} If $x=\mu(x)\in l_{\infty},$ then $(\mathbf{S}x)\prec\prec_{\log} 4(x\oplus x).$
\end{thm}
\begin{proof} Fix $n\geq0$ and denote
$$s=\frac{\prod_{m=0}^nx(m)}{x(n)^{n+1}}.$$
By Lemma \ref{ocenka s} \eqref{ocsa} and \eqref{ocsb}, we have
$$\prod_{k=0}^{2n+1}(\mathbf{S}x)(k)\leq\prod_{k=0}^nx(k)(1+\frac1{k+1}\log(s))\cdot\prod_{k=n+1}^{2n+1}x(n)(1+\frac1{k+1}\log(s))=$$
$$=(\prod_{k=0}^nx(k))^2\cdot s^{-1}\prod_{k=0}^{2n+1}(1+\frac1{k+1}\log(s))$$
and
$$\prod_{k=0}^{2n}(\mathbf{S}x)(k)\leq\prod_{k=0}^nx(k)(1+\frac1{k+1}\log(s))\cdot\prod_{k=n+1}^{2n}x(n)(1+\frac1{k+1}\log(s))=$$
$$=(\prod_{k=0}^{n-1}x(k))^2x(n)\cdot s^{-1}\prod_{k=0}^{2n}(1+\frac1{k+1}\log(s)).$$
Since $s\geq1,$ it follows from Lemma \ref{binomial} that
$$\prod_{k=0}^{2n+1}(1+\frac1{k+1}\log(s))\leq \prod_{k=0}^{2n+2}(1+\frac1{k+1}\log(s))\leq2^{2n+4+\log(s)}\leq 4^{2n+2}s$$
and
$$\prod_{k=0}^{2n}(1+\frac1{k+1}\log(s))\leq 2^{2n+2+\log(s)}\leq 4^{2n+1}s.$$
Therefore,
\begin{equation}\label{first eq}
\prod_{k=0}^{2n+1}(\mathbf{S}x)(k)\leq 4^{2n+2}(\prod_{k=0}^nx(k))^2=4^{2n+2}\prod_{k=0}^{2n+1}\mu(k,x\oplus x).
\end{equation}
and
\begin{equation}\label{second eq}
\prod_{k=0}^{2n}(\mathbf{S}x)(k)\leq 4^{2n+1}(\prod_{k=0}^{n-1}x(k))^2x(n)=4^{2n+1}\prod_{k=0}^{2n}\mu(k,x\oplus x).
\end{equation}
Since $n\geq0$ is arbitrary, the assertion follows by combining \eqref{first eq} and \eqref{second eq}.
\end{proof}

\section{Lidskii formula for traces on ideal closed with respect to the logarithmic submajorization}

In this section, we prove Theorems \ref{main commutator theorem} and \ref{extended kalton}.

The idea of the important estimate below belongs to Kalton (see Theorem 2.7 in \cite{Kalton1998}). We refer the reader to Chapter 5 in the book \cite{LSZ} for the detailed proof.

\begin{thm}\label{main quasinilpotent spectral estimate} If $Q$ is a compact quasi-nilpotent operator, then
$$|\sum_{|\lambda|>1,\lambda\in\sigma(\Re Q)}\lambda|\leq 400\sum_{|\lambda|>1,\lambda\in\sigma(2e|Q|)}\log(\lambda).$$
Here, $\Re Q$ is the real part of $Q.$ A similar assertion holds for the imaginary part $\Im Q.$
\end{thm}

\begin{lem}\label{prefinal commutator lemma} For every quasi-nilpotent compact operator $Q\in\mathcal{L}(H),$ we have
$$|C\lambda(\Re Q)|\leq 200\mathbf{S}((2eQ)^{\oplus 2}).$$
A similar assertion holds for $\Im Q.$
\end{lem}
\begin{proof} Since $\Re Q$ is self-adjoint, we infer from \eqref{mu sum} that, for every $n\geq0,$
$$|\lambda(n,\Re Q)|=\mu(n,\Re Q)=\mu(n,\frac12(Q+Q^*))\leq\mu(n,Q^{\oplus 2})\leq\mu(n,(2eQ)^{\oplus 2}).$$
Set
$$m(n)=\max\{m\geq0:\  |\lambda(m,\Re Q)|>\mu(n,2e(Q^{\oplus 2}))\},\quad n\geq0.$$
By the inequality above, we have $m(n)\leq n.$ Clearly,
$$\sigma(\Re Q)\cap\{\lambda:\ |\lambda|>\mu(n,2e(Q^{\oplus 2}))\}=\{\lambda(k,\Re Q)\}_{k=0}^{m(n)}$$
and
$$|\sum_{k=0}^n\lambda(k,\Re Q)|\leq\sum_{k=m(n)+1}^n|\lambda(k,\Re Q)|+|\sum_{k=0}^{m(n)}\lambda(k,\Re Q)|.$$
Therefore,
\begin{equation}\label{c est req}
|\sum_{k=0}^n\lambda(k,\Re Q)|\leq (n+1)\mu(n,(2eQ)^{\oplus 2})+|\sum_{\lambda\in\sigma(\Re Q),|\lambda|>\mu(n,(2eQ)^{\oplus 2})}\lambda|.
\end{equation}
In order to estimate the second summand at the right hand side of \eqref{c est req}, we define an operator $Q_0$ by setting
$$Q_0=\frac1{\mu(n,(2eQ)^{\oplus 2})}Q^{\oplus 2}.$$
It is clear that $\lambda(Q^{\oplus 2})=\sigma_2\lambda(Q)=0$ and, therefore, $Q_0$ is quasi-nilpotent operator. Applying Theorem \ref{main quasinilpotent spectral estimate}, we obtain
$$|\sum_{\lambda\in\sigma(\Re Q),|\lambda|>\mu(n,(2eQ)^{\oplus 2})}\lambda|=\frac12\mu(n,(2eQ)^{\oplus 2})|\sum_{\lambda\in\sigma(\Re Q_0),|\lambda|>1}\lambda|\leq$$
$$\leq200\mu(n,(2eQ)^{\oplus 2})\sum_{k=0}^n\log(2e\mu(k,Q_0))=200\mu(n,(2eQ)^{\oplus 2})\sum_{k=0}^n\log(\frac{\mu(k,(2eQ)^{\oplus 2})}{\mu(n,(2eQ)^{\oplus 2})}).$$
Using the latter estimate together with \eqref{c est req}, we obtain
$$|\sum_{k=0}^n\lambda(k,\Re Q)|\leq (n+1)\mu(n,(2eQ)^{\oplus 2})+200\mu(n,(2eQ)^{\oplus 2})\log(\frac{\prod_{k=0}^n\mu(k,(2eQ)^{\oplus 2})}{\mu(n,(2eQ)^{\oplus 2})^{n+1}}).$$
Dividing both sides by $(n+1)$ and appealing to the definition of $\mathbf{S}$ given in \eqref{s trans def} yields the assertion.
\end{proof}

The following proposition is the key to the proofs of Theorems \ref{main commutator theorem} and \ref{extended kalton}. Its proof crucially depends on Theorem \ref{hardest estimate}.

\begin{prop}\label{geom main estimate} For every quasi-nilpotent compact operator $Q\in\mathcal{L}(H),$ we have
$$C\lambda(\Re Q)\prec\prec_{\log} (1600eQ)^{\oplus 4},\quad C\lambda(\Im Q)\prec\prec_{\log} (1600eQ)^{\oplus 4}.$$
\end{prop}
\begin{proof} The assertion follows from consecutive application of Lemma \ref{prefinal commutator lemma} and Theorem \ref{hardest estimate}.
\end{proof}

We are now ready to the prove Theorems \ref{main commutator theorem} and \ref{extended kalton}.

\begin{proof}[of Theorem \ref{main commutator theorem}] Fix an operator $T\in\mathcal{I}.$ By Theorem \ref{ringrose theorem}, $T=N+Q,$ where $N$ is normal, $\lambda(N)=\lambda(T)$ and $Q$ is quasi-nilpotent. It follows from \eqref{lambda logmaj mu} that $\mu(N)=|\lambda(T)|\prec\prec_{\log} T.$ Since $\mathcal{I}$ is closed with respect to the logarithmic submajorization, it follows that $N\in\mathcal{I}.$ Thus, $Q=T-N$ also belongs to $\mathcal{I},$ and so does the operator $Q^{\oplus 4}.$

Since $\mathcal{I}$ is closed with respect to the logarithmic submajorization, it follows from Proposition \ref{geom main estimate} that $C\lambda(\Re Q)\in\mathcal{I}$ and $C\lambda(\Im Q)\in\mathcal{I}.$ It follows from Theorem \ref{normal commutator} that both $\Re Q\in{\rm Com}(\mathcal{I})$ and $\Im Q\in{\rm Com}(\mathcal{I}).$ Hence, $Q\in{\rm Com}(\mathcal{I}).$

The preceding paragraph yields the following conclusion: the operators $T$ and $N$ simultaneously belong or do not belong to ${\rm Com}(\mathcal{I}).$ By Theorem \ref{normal commutator}, the normal operator $N$ belongs to ${\rm Com}(\mathcal{I})$ if and only if $C\lambda(N)\in\mathcal{I}.$ Since $\lambda(N)=\lambda(T)$ (by construction), the assertion follows.
\end{proof}

\begin{proof}[of Theorem \ref{extended kalton}] Fix an operator $T\in\mathcal{I}.$ By Theorem \ref{ringrose theorem}, $T=N+Q,$ where $N$ is normal, $\lambda(N)=\lambda(T)$ and $Q$ is quasi-nilpotent. Repeating the argument in the proof of Theorem \ref{main commutator theorem}, we infer that $N\in\mathcal{I}$ and that $Q\in{\rm Com}(\mathcal{I}).$ It follows from the definition of trace that $\varphi(Q)=0.$ Hence, $\varphi(T)=\varphi(N).$ Since $N$ is normal, there exists an isometry $U\in\mathcal{L}(H)$ such that $N=U\lambda(N)U^*$ and $U^*U=1.$ Therefore, the operator
$$N-\lambda(N)=U\lambda(N)\cdot U^*-U^*\cdot U\lambda(N)\in{\rm Com}(\mathcal{I}).$$
It follows that $\varphi(N)=\varphi(\lambda(N))$ and, therefore, $\varphi(T)=\varphi(\lambda(T)).$ This proves \eqref{kalton first}.

The assertion of \eqref{kalton second} follows from Lemma \ref{dk prop11}.
\end{proof}

\section{Lidskii formula vs logarithmic submajorization}

The main aim of this section is to furnish the proof of Theorem \ref{main theorem}. The following definition provides an extension of a positive trace $\varphi$ given originally on the ideal $\mathcal{I}$ to the positive cone of the algebra $\mathcal{L}(H).$ We emphasize that $\varphi(A)$ does not have to be finite for a particular $0\leq A\in\mathcal{L}(H).$ In Theorem \ref{main theorem}, this obstacle is overcome by exploiting monotonicity of our trace with respect to the logarithmic submajorization.

\begin{defi}\label{extension def} Let $\mathcal{I}$ be an ideal in $\mathcal{L}(H)$ and let $\varphi$ be a positive trace on $\mathcal{I}.$ For every positive $A\in\mathcal{L}(H),$ set
\begin{equation}\label{extension eq}
\varphi(A)=\sup\{\varphi(B):\ 0\leq B\leq A,\ B\in\mathcal{I}\}.
\end{equation}
\end{defi}

The next lemma shows that in the Definition \ref{extension def}, the standard order can be replaced with the uniform submajorization.

\begin{lem}\label{ext maj} Let $\mathcal{I}$ be an ideal in $\mathcal{L}(H)$ and let $\varphi$ be a positive trace on $\mathcal{I}.$ For every positive $A\in\mathcal{L}(H),$ we have
$$\varphi(A)=\sup\{\varphi(B):\ 0\leq B\lhd A,\ B\in\mathcal{I}\}.$$
\end{lem}
\begin{proof} Fix $\varepsilon>0.$ It follows from Theorem \ref{ks majorization theorem} that there exist positive sequences $z_k\in l_{\infty},$ $1\leq k\leq n,$ and positive constants $\lambda_k,$ $1\leq k\leq n,$  such that $\mu(z_k)\leq\mu(A)$ and
$$(1-\varepsilon)\mu(B)=\sum_{k=1}^n\lambda_kz_k,\quad\sum_{k=1}^n\lambda_k=1.$$
Without loss of generality, $\lambda_k>0$ for all $1\leq k\leq n.$ It follows that $z_k\leq\lambda_k^{-1}\mu(B)$ and, therefore, $z_k\in\mathcal{I}$ for all $1\leq k\leq n.$ By Definition \ref{extension def}, we have $\varphi(z_k)\leq\varphi(A).$ Hence,
$$(1-\varepsilon)\varphi(B)=\sum_{k=1}^n\lambda_k\varphi(z_k)\leq\varphi(A).$$
Since $\varepsilon>0$ is arbitrarily small, the assertion follows.
\end{proof}

The following lemma shows that the extension of $\varphi$ given in Definition \ref{extension def} is additive with respect to the direct sum operation.

\begin{lem}\label{direct sum} Let $\mathcal{I}$ be an ideal in $\mathcal{L}(H)$ and let $\varphi$ be a positive trace on $\mathcal{I}.$ For every positive operators $A_1,A_2\in\mathcal{L}(H),$ we have
$$\varphi(A_1\oplus A_2)=\varphi(A_1)+\varphi(A_2).$$
\end{lem}
\begin{proof} If $B_1,B_2\in\mathcal{I}$ are such that $0\leq B_1\leq A_1$ and $0\leq B_2\leq A_2,$ then $0\leq B_1\oplus B_2\leq A_1\oplus A_2.$ Hence,
$$\sup\{\varphi(B_1\oplus B_2):\ 0\leq B_1\leq A_1,\ 0\leq B_2\leq A_2,\ B_1,B_2\in\mathcal{I}\}\leq$$
$$\leq\sup\{\varphi(B):\ 0\leq B\leq A_1\oplus A_2,\ B\in\mathcal{I}\}.$$
It follows from Definition \ref{extension def} that
\begin{equation}\label{fi plus diskoint1}
\varphi(A_1)+\varphi(A_2)\leq\varphi(A_1\oplus A_2).
\end{equation}

In order to prove the converse inequality, let $B\in\mathcal{I}$ be such that $0\leq B\leq A_1\oplus A_2.$ Let $p$ be the support projection of $A_1\oplus 0$ and let $U=p+i(1-p).$ Since $U$ is unitary, it follows that $\varphi(C)=\varphi(U^{-1}CU)$ for every operator $C\in\mathcal{I}.$ Note that
$$Bp-U^{-1}(Bp)U=Bp-U^{-1}Bp=(1+i)(1-p)Bp.$$
Therefore, $\varphi((1-p)Bp)=0$ and, similarly, $\varphi(pB(1-p))=0.$ Hence,
$$\varphi(B)=\varphi(pBp)+\varphi((1-p)B(1-p)).$$
On the other hand, we have
$$pBp\leq p(A_1\oplus A_2)p=A_1\oplus 0,\quad(1-p)B(1-p)\leq(1-p)(A_1\oplus A_2)(1-p)=0\oplus A_2.$$
Setting $B_1=pBp$ and $B_2=(1-p)B(1-p),$ we have $\varphi(B)=\varphi(B_1+B_2)$ and $0\leq B_1\leq A_1\oplus0,$ $0\leq B_2\leq 0\oplus A_2.$ Therefore,
$$\varphi(A)=\sup\{\varphi(B):\ 0\leq B\leq A_1\oplus A_2,\ B\in\mathcal{I}\}\leq$$
$$\leq\sup\{\varphi(B_1)+\varphi(B_2):\ 0\leq B_1\leq A_1\oplus 0,\ 0\leq B_2\leq 0\oplus A_2,\ B_1,B_2\in\mathcal{I}\}=$$
$$=\sup\{\varphi(B_1):\ 0\leq B_1\leq A_1,\ B_1\in\mathcal{I}\}+\sup\{\varphi(B_2):\ 0\leq B_2\leq A_2,\ B_2\in\mathcal{I}\}.$$
Thus, by Definition \ref{extension def}, we have
\begin{equation}\label{fi plus diskoint2}
\varphi(A_1\oplus A_2)\geq\varphi(A_1)+\varphi(A_2).
\end{equation}
The assertion follows from \eqref{fi plus diskoint1} and \eqref{fi plus diskoint2}.
\end{proof}

Finally, we are prepared to show that the extension of a positive trace $\varphi$ is a trace on every ideal $\mathcal{J}$ on which this extension takes finite values.

\begin{prop}\label{extension lemma} Let $\mathcal{I}$ be an ideal in $\mathcal{L}(H)$ and let $\varphi$ be a positive trace on $\mathcal{I}.$ If an ideal $\mathcal{J}\supset\mathcal{I}$ is such that $\varphi(A)<\infty$ for every $0\leq A\in\mathcal{J}$ (where $\varphi(A)$ is as in Definition \ref{extension def}), then $\varphi$ is a positive trace on $\mathcal{J}.$
\end{prop}
\begin{proof} Let $0\leq A_1,A_2\in\mathcal{J}.$ Applying the inequality \eqref{uniform maj sum} to the positive operators $A_1\oplus 0$ and $0\oplus A_2$ (and, separately, to the positive operators $A_1,A_2$) we obtain
$$A_1\oplus A_2\lhd\mu(A_1)+\mu(A_2)\lhd 2\sigma_{1/2}\mu(A_1+A_2)$$
and
$$A_1+A_2\lhd\mu(A_1)+\mu(A_2)\lhd2\sigma_{1/2}\mu(A_1\oplus A_2).$$
By Lemma \ref{ext maj}, we have
$$\varphi(A_1\oplus A_2)\leq\varphi(2\sigma_{1/2}\mu(A_1+A_2))=\varphi(A_1+A_2)$$
and
$$\varphi(A_1+A_2)\leq\varphi(2\sigma_{1/2}\mu(A_1\oplus A_2))=\varphi(A_1\oplus A_2).$$
It follows now from Lemma \ref{direct sum} that
$$\varphi(A_1+A_2)=\varphi(A_1\oplus A_2)=\varphi(A_1)+\varphi(A_2).$$

Hence, $\varphi$ extends to a positive linear functional on $\mathcal{J}.$ By Lemma \ref{trace monotone}, this functional is a trace.
\end{proof}

The following inequalities, though important {\it per se}, play a crucial role in the proof of the subsequent Proposition \ref{geom envelope lemma}.

\begin{lem}\label{sum lessdot} Let $A_k,B_k\in\mathcal{L}(H)$ $k=1,2,$ be such that $B_k\prec\prec_{\log} A_k.$ We have
\begin{enumerate}[(a)]
\item\label{suma} $$B_1\oplus B_2\prec\prec_{\log} A_1\oplus A_2.$$
\item\label{sumb} $$B_1+B_2\prec\prec_{\log} 2(A_1\oplus A_2)^{\oplus 2}.$$
\end{enumerate}
\end{lem}
\begin{proof} In fact, second assertion follows easily from the first one.
\begin{enumerate}[(a)]
\item We have
$$\prod_{k=0}^n\mu(k,B_1\oplus B_2)=\sup_{n_1+n_2=n+1}\prod_{k=0}^{n_1-1}\mu(k,B_1)\prod_{k=0}^{n_2-1}\mu(k,B_2)\leq$$
$$\leq\sup_{n_1+n_2=n+1}\prod_{k=0}^{n_1-1}\mu(k,A_1)\prod_{k=0}^{n_2-1}\mu(k,A_2)=\prod_{k=0}^n\mu(k,A_1\oplus A_2).$$
\item It follows from obvious inequalities
$$\mu(B_1)\leq\mu(B_1\oplus B_2),\quad \mu(B_2)\leq\mu(B_1\oplus B_2)$$
combined with \eqref{mu sum} and part \eqref{suma} that
$$\mu(B_1+B_2)\leq 2\sigma_2\mu(B_1\oplus B_2)\prec\prec_{\log} 2\sigma_2\mu(A_1\oplus A_2)=2\mu((A_1\oplus A_2)^{\oplus 2}).$$
\end{enumerate}
\end{proof}

\begin{prop}\label{geom envelope lemma} For every ideal $\mathcal{I},$ there exists the least ideal $LE(\mathcal{I})$ containing $\mathcal{I}$ and closed with respect to the logarithmic submajorization. This ideal can be defined as follows
\begin{equation}\label{le def}
LE(\mathcal{I})=\{B\in\mathcal{L}(H):\ B\prec\prec_{\log} A\mbox{ for some }A\in\mathcal{I}\}.
\end{equation}
\end{prop}
\begin{proof} If $B_1,B_2\in LE(\mathcal{I}),$ then there exist $A_1,A_2\in\mathcal{I}$ such that $B_1\prec\prec_{\log} A_1$ and $B_2\prec\prec_{\log} A_2.$ We have $A_1\oplus A_2\in\mathcal{I}$ and, therefore, $(A_1\oplus A_2)^{\oplus 2}\in\mathcal{I}.$ It follows from Lemma \ref{sum lessdot} that $B_1+B_2\prec\prec_{\log} 2(A_1\oplus A_2)^{\oplus 2}\in\mathcal{I}.$ In particular, $B_1+B_2\in LE(\mathcal{I}).$ Thus, $LE(\mathcal{I})$ is a linear space.

We now show that $LE(\mathcal{I})$ is an ideal in $\mathcal{L}(H).$ Indeed, for every $B\in LE(\mathcal{I}),$ there exists an operator $A\in\mathcal{I}$ such that $B\prec\prec_{\log}A.$ If now $C\in\mathcal{L}(H),$ then
$$\mu(BC)\leq\|C\|_{\infty}\mu(B)\prec\prec_{\log}\|C\|_{\infty}A\in\mathcal{I}.$$
Hence, $BC\in LE(\mathcal{I}).$ Similarly, $CB\in LE(\mathcal{I})$ and, therefore, $LE(\mathcal{I})$ is an ideal in $\mathcal{L}(H).$

It is clear that $LE(\mathcal{I})$ is closed with respect to the logarithmic submajorization and that every ideal closed with respect to the logarithmic submajorization and containing $\mathcal{I}$ must also contain $LE(\mathcal{I}).$
\end{proof}

\begin{lem}\label{abs lemma} Let $\varphi$ be a positive trace on the ideal $\mathcal{I}.$ For every $T\in\mathcal{I},$ we have $|\varphi(\Re T)|\leq\varphi(|T|).$
\end{lem}
\begin{proof} By Lemma 4.3 of \cite{FackKosaki}, there exist partial isometries $U,V\in\mathcal{L}(H)$ such that
$$|\Re T|\leq \frac12(U|T|U^*+V|T^*|V^*).$$
Note that
$$\mu(U|T|U^*)\leq\mu(|T|),\quad \mu(V|T^*|V^*)\leq\mu(|T^*|)=\mu(|T|).$$
Since $\varphi$ is a positive trace, it follows from Lemma \ref{trace monotone} that
$$|\varphi(\Re T)|\leq\varphi(|\Re T|)\leq\frac12(\varphi(U|T|U^*)+\varphi(V|T^*|V^*))\leq\varphi(|T|).$$
\end{proof}

We are now ready to prove the second main result of the paper.

\begin{proof}[of Theorem \ref{main theorem}] First, we prove \eqref{main first}. Let $\varphi$ be a positive spectral trace on the ideal $\mathcal{I}$ and let positive operators $A,B\in\mathcal{I}$ be such that $B\prec\prec_{\log}A.$ By Lemma \ref{dk prop11}, there exists $T\in\mathcal{L}(H)$ such that $\lambda(T)=\mu(B)$ and $\mu(T)\leq\mu(A).$ By the assumption, the trace $\varphi$ is spectral and, therefore, we have $\varphi(B)=\varphi(\lambda(T))=\varphi(T).$ In particular, $\varphi(T)\in\mathbb{R}$ and, therefore, $\varphi(T)=\varphi(\Re T).$ By Lemma \ref{abs lemma}, we have $\varphi(B)=\varphi(\Re T)\leq\varphi(|T|).$ Since $\mu(T)\leq\mu(A),$ it follows that $\varphi(B)\leq\varphi(|T|)\leq\varphi(A).$ This proves \eqref{main first}.

Now, we prove \eqref{main second}. Let $\varphi$ be a positive trace on $\mathcal{I}$ which is monotone with respect to the logarithmic submajorization. If $0\leq C\in LE(\mathcal{I}),$ then there exists $A\in\mathcal{I}$ such that $C\prec\prec_{\log} A.$ If now $0\leq B\in\mathcal{I}$ is such that $B\leq C,$ then, obviously, $B\prec\prec_{\log} A.$ Hence, by the assumption, $\varphi(B)\leq\varphi(A).$ If we define $\varphi$ on $LE(\mathcal{I})$ by formula \eqref{extension eq}, then $\varphi(C)\leq\varphi(A)<\infty.$ By Proposition \ref{extension lemma}, $\varphi$ is a positive trace on $LE(\mathcal{I}).$ By Theorem \ref{extended kalton} \eqref{kalton first}, $\varphi(T)=\varphi(\lambda(T))$ for every $T\in\mathcal{I}.$ This proves \eqref{main second}.
\end{proof}

\section{Examples}\label{Examples}

The following lemma shows that every geometrically stable ideal is closed with respect to the logarithmic submajorization. Theorem \ref{dk example} \eqref{dkc} shows that the converse is false.

\begin{lem}\label{gs implies cl} Every geometrically stable ideal is closed with respect to the logarithmic submajorization.
\end{lem}
\begin{proof} Let $A\in\mathcal{I}$ and let $B\in\mathcal{L}(H)$ be such that $B\prec\prec_{\log}A.$ We have
$$\mu(n,B)\leq(\prod_{k=0}^n\mu(k,B))^{1/(n+1)}\leq(\prod_{k=0}^n\mu(k,A))^{1/(n+1)},\quad n\geq0.$$
The assertion follows now from the definition of a geometrically stable ideal.
\end{proof}

\begin{thm}\label{dk example} Define an operator $A\in\mathcal{L}(H)$ by setting\footnote{That is, $\mu(k,A)=2^{-2^{3n}}$ for all $k\in[2^{2^{3(n-1)}},2^{2^{3n}}).$}
$$\mu(A)=\sup_{n\geq0}2^{-2^{3n}}\chi_{[0,2^{2^{3n}})}.$$
For the principal ideal $\mathcal{I}_A$ generated by $A,$ we have
\begin{enumerate}[(a)]
\item\label{dka} Every positive trace $\varphi$ on $\mathcal{I}_A$ extends to a positive trace on $LE(\mathcal{I}_A).$ In particular, $\varphi$ is spectral and monotone with respect to logarithmic submajorization.
\item\label{dkb} Ideal $\mathcal{I}_A$ is not closed with respect to the logarithmic submajorization. Moreover, there exists a trace (non-positive) on $\mathcal{I}$ which is not spectral.
\item\label{dkc} The ideal $LE(\mathcal{I}_A)$ fails to be geometrically stable.
\end{enumerate}
\end{thm}

Lemma \ref{a0 vanish} and Proposition \ref{horrible technical estimate} below are needed for the proof of Theorem \ref{dk example} \eqref{dka}.

\begin{lem}\label{a0 vanish} Let $A$ be as in Theorem \ref{dk example} and let a compact operator $A_0\in\mathcal{L}(H)$ be defined by setting
$$\mu(A_0)=\sup_{n\geq0}2^{-2^{3(n+1)}}\chi_{[0,2^{3n+2^{3n}})}.$$
We have $A_0\in\mathcal{I}_A$ and $\varphi(A_0)=0$ for every positive trace on $\mathcal{I}_A.$
\end{lem}
\begin{proof} Clearly, $\mu(A_0)\leq\mu(A)\in\mathcal{I}_A.$ For every $l\geq0,$ we have
$$\sigma_{2^l}\mu(A_0)=\sup_{n\geq0}2^{-2^{3(n+1)}}\chi_{[0,2^{3n+l+2^{3n}})}.$$
Note that $3n+l+2^{3n}\leq 2^{3(n+1)}$ for $n\geq l.$ It follows that
$$\sigma_{2^l}\mu(A_0)\leq\sup_{0\leq n<l}2^{-2^{3(n+1)}}\chi_{[0,2^{3n+l+2^{3n}})}+\sup_{n\geq 0}2^{-2^{3(n+1)}}\chi_{[0,2^{2^{3(n+1)}})}.$$
Hence,
$$\sigma_{2^l}\mu(A_0)\leq\sigma_{2^l}\sup_{0\leq n<l}2^{-2^{3(n+1)}}\chi_{[0,2^{3n+2^{3n}})}+\mu(A).$$
Since $\varphi$ is a positive trace, it follows from Lemma \ref{trace monotone} that
$$2^l\varphi(A_0)\leq 2^l\varphi(\sup_{0\leq n<l}2^{-2^{3(n+1)}}\chi_{[0,2^{3n+2^{3n}})})+\varphi(A).$$
By Lemma \ref{singular} \eqref{singa}, the first term at the right hand side is $0.$ Therefore, $\varphi(A_0)\leq 2^{-l}\varphi(A).$ Since $l$ is arbitrarily large, the assertion follows.
\end{proof}

\begin{prop}\label{horrible technical estimate} Let $A$ be as in Theorem \ref{dk example} and let $A_0$ be as in Lemma \ref{a0 vanish}. If $B\in\mathcal{I}_A$ is such that $B\prec\prec_{\log} A,$ then there exists $l\geq0$ such that
\begin{equation}\label{horror}
\mu(B)\leq 2^l\mu(A_0)+256\sigma_2\mu(A)+\sigma_{2^l}(\mu^4(A)).
\end{equation}
\end{prop}
\begin{proof} Since $\mathcal{I}_A$ is a principal ideal and since $B\in\mathcal{I}_A,$ it follows from \eqref{mu sum} that there exists $l\geq1$ such that $\mu(B)\leq 2^l\sigma_{2^l}\mu(A).$ We verify the inequality \eqref{horror} on the following intervals.
$$[2^{l+2^{3n}},2^{3n+2^{3n}}),\quad [2^{1+2^{3n}},2^{l+2^{3n}}),\quad [2^{3n+2^{3n}},2^{2^{3(n+1)}}),\quad [2^{2^{3n}},2^{1+2^{3n}}).$$
Here, it is presumed that, when $3n\leq l,$ we do not consider first interval at all.

\noindent For every $k\in[2^{l+2^{3n}},2^{3n+2^{3n}}),$ it follows from the definition of $\sigma_{2^l}$ that
$$\mu(k,B)\leq (2^l\sigma_{2^l}\mu(A))(k)=2^l\mu(k,A)=2^l\mu(k,A_0).$$
For every $k\in [2^{1+2^{3n}},2^{l+2^{3n}}),$ it follows from the assumption $B\prec\prec_{\log} A$ that
$$\mu(k,B)\leq(\prod_{m=0}^{2^{1+2^{3n}}-1}\mu(m,B))^{2^{-1-2^{3n}}}\leq(\prod_{m=0}^{2^{1+2^{3n}}-1}\mu(m,A))^{2^{-1-2^{3n}}}\leq$$
$$\leq(\prod_{m=2^{2^{3n}}}^{2^{1+2^{3n}}-1}\mu(m,A))^{2^{-1-2^{3n}}}=2^{-4\cdot 2^{3n}}=(\sigma_{2^l}\mu^4(A))(k).$$
For every $k\in[2^{3n+2^{3n}},2^{2^{3(n+1)}}),$ it follows from the assumption $B\prec\prec_{\log} A$ that
$$\mu(k,B)\leq(\prod_{m=0}^k\mu(m,B))^{1/(k+1)}\leq(\prod_{m=0}^k\mu(m,A))^{1/(k+1)}\leq$$
$$\leq(\prod_{m=2^{2^{3n}}}^k\mu(m,A))^{1/(k+1)}=(2^{-2^{3(n+1)}})^{(k+1-2^{2^{3n}})/(k+1)}\leq2^{8-2^{3(n+1)}}=256\mu(k,A).$$
Arguing similarly, we infer that, for every $k\in[2^{2^{3n}},2^{1+2^{3n}}),$ we have
$$\mu(k,B)\leq(\prod_{m=0}^{2^{2^{3n}}-1}\mu(m,B))^{2^{-2^{3n}}}\leq(\prod_{m=0}^{2^{2^{3n}}-1}\mu(m,A))^{2^{-2^{3n}}}\leq$$
$$\leq(\prod_{m=2^{2^{3(n-1)}}}^{2^{2^{3n}}-1}\mu(m,A))^{2^{-2^{3n}}}=2^{-2^{3n}(2^{2^{3n}}-2^{2^{3(n-1)}})2^{-2^{3n}}}\leq 2^{8-2^{3n}}=(256\sigma_2\mu(A))(k).$$
A combination of all $4$ preceding estimates yields the assertion.
\end{proof}

Define the nonlinear homogeneous mapping $\mathbf{T}:l_{\infty}\to l_{\infty}$ by setting
\begin{equation}\label{t def}
(\mathbf{T}x)(k)=(\prod_{m=0}^k\mu(m,x))^{1/(k+1)},\quad k\geq0.
\end{equation}

Lemma \ref{t properties}, Lemma \ref{t aux lemma} and Lemma \ref{t main lemma} below are used in the proof of Theorem \ref{dk example} \eqref{dkc}.

\begin{lem}\label{t properties} Operator $\mathbf{T}$ has the following properties.
\begin{enumerate}[(a)]
\item\label{tpa} for every $x\in l_{\infty},$ $\mathbf{T}x=\mu(\mathbf{T}x).$
\item\label{tpb} for every $x,y\in l_{\infty},$ we have $y\prec\prec_{\log}x$ if and only if $\mathbf{T}y\leq\mathbf{T}x.$
\item\label{tpc} for every $x\in l_{\infty}$ we have $\sigma_N\mathbf{T}x\leq\mathbf{T}(\sigma_Nx)\leq\sigma_{2N}\mathbf{T}x.$
\item\label{tpd} for every $x\in l_{\infty},$ $\mathbf{T}x\prec\prec_{\log}N\sigma_Nx$ implies that $\mathbf{T}^2x\leq N\sigma_{2N}\mathbf{T}x.$
\item\label{tpe} for every $x\in l_{\infty},$ $\mathbf{T}^2x\leq N\sigma_N\mathbf{T}x$ implies that $\mathbf{T}x\prec\prec_{\log}N\sigma_Nx.$
\end{enumerate}
\end{lem}
\begin{proof} The claims \eqref{tpa} and \eqref{tpb} are obvious.

In order to prove the claim \eqref{tpc}, let $m\geq0$ and let $0\leq r\leq N-1.$ It follows from the first assertion that
$$(\prod_{k=0}^m\mu(k,x))^{\frac1{m+1}}=(\prod_{k=0}^{(m+1)N-1}\mu(k,\sigma_Nx))^{\frac1{(m+1)N}}\leq(\prod_{k=0}^{mN+r}\mu(k,\sigma_Nx))^{\frac1{mN+r+1}}.$$
Equivalently, (setting $n=mN+r$), we have
$$(\prod_{k=0}^{[\frac{n}{N}]}\mu(k,x))^{1/([\frac{n}{N}]+1)}\leq(\prod_{k=0}^n\mu(k,\sigma_Nx))^{1/(n+1)}.$$
Since $(\sigma_Nz)(n)=z([\frac{n}{N}])$ for all $z\in l_{\infty},$ we rewrite the preceding inequality as $\sigma_N\mathbf{T}x\leq\mathbf{T}(\sigma_Nx).$ The proof of the inequality $\mathbf{T}(\sigma_Nx)\leq\sigma_{2N}\mathbf{T}x$ is similar.

The claim \eqref{tpd} follows from consecutive application of \eqref{tpb} and \eqref{tpc}. The claim \eqref{tpe} assertions follows from consecutive application of \eqref{tpc} and \eqref{tpb}.
\end{proof}

\begin{lem}\label{t aux lemma} For every $2^{2^{3n}}\leq k<2^{2^{3(n+1)}},$ we have
$$2^{\frac{7\gamma_n}{k+1}}\leq2^{2^{3(n+1)}}(\mathbf{T}\mu(A))(k)\leq 2^{1+\frac{7\gamma_n}{k+1}}.$$
Here, $\gamma_n=2^{3n+2^{3n}}.$
\end{lem}
\begin{proof} Fix $k$ and $n$ satisfying the assumption. By the definition of $A,$ we have
$$\prod_{m=0}^k\mu(m,A)=\frac14\prod_{s=1}^n\prod_{m=2^{2^{3(s-1)}}}^{2^{2^{3s}}-1}\mu(m,A)\cdot\prod_{m=2^{2^{3n}}}^k\mu(m,A)=$$
$$=\frac14\cdot\prod_{s=1}^n2^{-2^{3s}(2^{2^{3s}}-2^{2^{3(s-1)}})}\cdot 2^{-2^{3(n+1)}(k+1-2^{2^{3n}})}.$$
Hence, multiplying both sides by $2^{(k+1)2^{3(n+1)}}$ and using the abbreviation $\gamma_n=2^{3n+2^{3n}},$ we obtain
$$2^{(k+1)2^{3(n+1)}}\prod_{m=0}^k\mu(m,A)=\frac14\prod_{s=1}^n2^{8\gamma_{s-1}-\gamma_s}\cdot 2^{8\gamma_n}=2^{14}\cdot\prod_{s=1}^n2^{7\gamma_s}.$$
The assertion follows now from the inequality
$$2^{7\gamma_n}\leq\prod_{s=1}^n2^{7\gamma_s}\leq 2^{7\gamma_n+7(n-1)\gamma_{n-1}}\leq 2^{7\gamma_n+k-13},$$
where we used the estimate $7(n-1)\gamma_{n-1}\leq k-13,$ which holds for every $n$ and for every $k\in[2^{2^{3n}},2^{2^{3(n+1)}}).$
\end{proof}

\begin{lem}\label{t main lemma} For every $l\geq 1,$ the inequality $\mathbf{T}^2\mu(A)\leq 2^l\sigma_{2^l}\mathbf{T}\mu(A)$ fails.
\end{lem}
\begin{proof} We will demonstrate that the inequality above fails at the point $\gamma_n-1$ for all sufficiently large $n$ (we use the abbreviation $\gamma_n=2^{3n+2^{3n}}$ from Lemma \ref{t aux lemma}). It is obvious that
$$(\mathbf{T}^2\mu(A))(\gamma_n-1)=2^{-2^{3(n+1)}}(\prod_{m=0}^{\gamma_n-1}2^{2^{3(n+1)}}(\mathbf{T}\mu(A))(m))^{1/\gamma_n}.$$
Using the fact (guaranteed by the inequality $\mathbf{T}\mu(A)\geq\mu(A)$ and Lemma \ref{t aux lemma}) that
$$2^{2^{3(n+1)}}(\mathbf{T}\mu(A))(m)\geq\begin{cases}
1,&m<2^{2^{3n}}\\
2^{\frac{7\gamma_n}{m+1}},&2^{2^{3n}}\leq m<2^{2^{3(n+1)}}
\end{cases}
$$
we infer that
$$(\mathbf{T}^2\mu(A))(\gamma_n-1)\geq 2^{-2^{3(n+1)}}
(\prod_{m=2^{2^{3n}}}^{\gamma_n-1}2^{\frac{7\gamma_n}{m+1}})^{1/\gamma_n}=2^{-2^{3(n+1)}}\cdot 2^{\sum_{m=2^{2^{3n}}}^{\gamma_n-1}\frac7{m+1}}.$$
Since
$$\sum_{m=2^{2^{3n}}}^{\gamma_n-1}\frac1{m+1}\geq\log(\frac{\gamma_n}{2^{2^{3n}}})-2^{-2^{3n}}\geq n,$$
it follows that
$$(\mathbf{T}^2\mu(A))(\gamma_n-1)\geq 2^{7n-2^{3(n+1)}}.$$
On the other hand, it follows from Lemma \ref{t aux lemma} (with $k=2^{3n-l+2^{3n}}-1$) that
$$(2^l\sigma_{2^l}\mathbf{T}\mu(A))(\gamma_n-1)=2^l(\mathbf{T}\mu(A))(2^{3n-l+2^{3n}}-1)\leq 2^{l+1}\cdot 2^{7\cdot 2^l-2^{3(n+1)}}.$$
Thus, for $n\geq 2^{l+1},$ we have
$$(\mathbf{T}^2\mu(A))(\gamma_n-1)\geq(2^l\sigma_{2^l}\mathbf{T}\mu(A))(\gamma_n-1).$$
\end{proof}

We are now ready to prove main result of this section.

\begin{proof}[of Theorem \ref{dk example}] First, we prove \eqref{dka}. In what follows, $A_0$ is as in Lemma \ref{a0 vanish}. Let $0\leq B\in\mathcal{I}_A$ is such that $B\prec\prec_{\log} A.$ If $\varphi$ is a positive trace on $\mathcal{I}_A,$ then we infer from Proposition \ref{horrible technical estimate} that
$$\varphi(B)\leq512\varphi(A)+2^l\varphi(\mu^4(A))+2^l\varphi(A_0).$$
By Lemma \ref{a0 vanish}, $\varphi(A_0)=0.$ Since $\mu^4(k,A)=o(\mu(k,A))$ as $k\to\infty,$ it follows from Lemma \ref{singular} that $\varphi(\mu^4(A))=0.$ Hence, $\varphi(B)\leq 512\varphi(A).$ Repeating the argument in the proof of Theorem \ref{main theorem} \eqref{main second}, we obtain that $\varphi$ extends to a positive trace on $LE(\mathcal{I}_A).$ We keep denoting this extension by $\varphi.$ By Theorem \ref{extended kalton} \eqref{kalton first}, the positive trace $\varphi$ on the ideal $LE(\mathcal{I}_A)$ is spectral. Applying Theorem \ref{main theorem} \eqref{main first}, we infer that $\varphi$ is monotone with respect to the logarithmic submajorization. Thus, the original trace $\varphi$ on $\mathcal{I}$ is also monotone with respect to the logarithmic submajorization.

We now turn to \eqref{dkb}. This assertion is, in fact, proved in Example 1.5 of \cite{DykemaKalton1998}. More precisely, it is shown there that there exists a quasi-nilpotent operator $Q\in\mathcal{I}_A$ such that $Q\notin{\rm Com}(\mathcal{I}_A).$ By Zorn lemma, there exists a linear functional $\varphi$ on $\mathcal{I}_A$ such that $\varphi(Q)=1$ and such that $\varphi$ vanishes on ${\rm Com}(\mathcal{I}_A).$ Hence, $\varphi$ is a (non-positive) trace on $\mathcal{I}_A$ such that $\varphi(Q)=1.$ Since $Q$ is quasi-nilpotent, it follows that $\lambda(Q)=0$ and, therefore, $\varphi(\lambda(Q))=0.$ However, $\varphi(Q)=1$ by construction. Hence, $\varphi$ is not spectral. This proves \eqref{dkb}.

Finally, we prove \eqref{dkc}. More precisely, we will prove that
$$\{(\prod_{m=0}^n\mu(k,A))^{1/(n+1))}\}_{n\geq0}\notin LE(\mathcal{I}_A).$$
In terms of the operator $\mathbf{T}$ introduced in \eqref{t def}, the assertion above can be written as $\mathbf{T}\mu(A)\notin LE(\mathcal{I}_A).$ Assume the contrary. It follows from \eqref{le def} that there exists $B\in\mathcal{I}_A$ such that $\mathbf{T}\mu(A)\prec\prec_{\log}B.$ Since $\mathcal{I}_A$ is a principal ideal, it follows that $\mu(B)\leq 2^l\sigma_{2^l}\mu(A)$ for some $l\geq1.$ Hence, $\mathbf{T}\mu(A)\prec\prec 2^l\sigma_{2^l}\mu(A).$ By Lemma \ref{t properties} \eqref{tpd}, we have $\mathbf{T}^2\mu(A)\leq 2^l\sigma_{2^{l+1}}\mathbf{T}\mu(A).$ However, the latter inequality fails by Lemma \ref{t main lemma}. This proves \eqref{dkc}.
\end{proof}

\end{document}